\UseRawInputEncoding
\documentclass[12pt]{article}

\usepackage{dsfont}
\usepackage{amsfonts,amsmath,amsthm}
\usepackage{algorithm, algorithmic}
\usepackage{color}
\usepackage{graphicx}
\usepackage{stmaryrd}        

\usepackage{caption}
\usepackage{subcaption}

\usepackage[paper=a4paper,dvips,top=2cm,left=2cm,right=2cm,
foot=1cm,bottom=4cm]{geometry}

\newcommand{\dc}[1]{#1}

\newcommand{\vdc}[1]{\mathbf{#1}}

\newcommand{\dcm}[1]{#1}
\newcommand{\rset}{\mathbb R}
\newcommand{\drset}{\hat{\mathbb R}}
\newcommand{\qset}{\mathbb H}
\newcommand{\dqset}{\hat{\mathbb H}}
\newcommand{\cset}{\mathbb C}
\newcommand{\dcset}{\hat{\mathbb C}}
\newcommand{\gset}{\mathbb G}
\newcommand{\dgset}{\hat{\mathbb G}}
\newcommand{\cgroup}{{\mathbb C}^{\times}}

\newcommand{\ucset}{\mathbb T}
\newcommand{\uqset}{\mathbb U}
\newcommand{\udcset}{\hat{\mathbb T}}
\newcommand{\udqset}{\hat{\mathbb U}}
\newcommand{\udset}{\hat{\mathbb V}}

\begin{document}
	\large
	
	\title{{Spectral Properties of Dual  Unit Gain Graphs}}
	\author{Chunfeng Cui\footnote{LMIB of the Ministry of Education, School of Mathematical Sciences, Beihang University, Beijing 100191, China.
			({\tt chunfengcui@buaa.edu.cn}). }
		\and \
		Yong Lu\footnote{School of Mathematics and Statistics, Jiangsu Normal University, Xuzhou, Jiangsu 221116, China.
			({\tt luyong@jsnu.edu.cn}).}
		\and \
		Liqun Qi\footnote{Department of Applied Mathematics, The Hong Kong Polytechnic University, Hung Hom, Kowloon, Hong Kong; Department of Mathematics, School of Science, Hangzhou Dianzi University, Hangzhou 310018, China.
			({\tt maqilq@polyu.edu.hk}).}
		\and and \
		Ligong  Wang\footnote{School of Mathematics and Statistics, Northwestern Polytechnical University, Xi'an, Shaanxi 710129, China.
			({\tt lgwangmath@163.com}).}
	}
	\date{\today}
	\maketitle

	\begin{abstract}
		
		In this paper, we study dual quaternion and dual complex unit gain graphs and their spectral properties  in a unified frame of  dual  unit gain graphs.    
		{Unit dual quaternions represent rigid movements in the 3D space, and have wide applications in  robotics and computer graphics.   Dual complex numbers found application in brain science recently.}
		We establish the interlacing theorem for dual   unit gain graphs, and show that the spectral radius of a dual  unit gain graph is always not greater than the spectral radius of the underlying graph, and these two radii are equal if and only if the dual  gain graph is balanced.  By using the dual  cosine functions,  we establish the closed form of eigenvalues of adjacency and Laplacian matrices of dual complex  and  quaternion unit gain cycles.   We then show the  coefficient theorem holds for dual unit gain graphs.  Similar results hold for the spectral radius of the Laplacian matrix of the dual  unit gain graph too.

		\medskip


		
	\end{abstract}

	\renewcommand{\Re}{\mathds{R}}
	\newcommand{\rank}{\mathrm{rank}}
	\renewcommand{\span}{\mathrm{span}}
	\newcommand{\X}{\mathcal{X}}
	\newcommand{\A}{\mathcal{A}}
	\newcommand{\I}{\mathcal{I}}
	\newcommand{\B}{\mathcal{B}}
	\newcommand{\C}{\mathcal{C}}
	\newcommand{\OO}{\mathcal{O}}
	\newcommand{\e}{\mathbf{e}}
	\newcommand{\0}{\mathbf{0}}
	\newcommand{\dd}{\mathbf{d}}
	\newcommand{\ii}{\mathbf{i}}
	\newcommand{\jj}{\mathbf{j}}
	\newcommand{\kk}{\mathbf{k}}
	\newcommand{\va}{\mathbf{a}}
	\newcommand{\vb}{\mathbf{b}}
	\newcommand{\vc}{\mathbf{c}}
	\newcommand{\vq}{\mathbf{q}}
	\newcommand{\vg}{\mathbf{g}}
	\newcommand{\pr}{\vec{r}}
	\newcommand{\pc}{\vec{c}}
	\newcommand{\ps}{\vec{s}}
	\newcommand{\pt}{\vec{t}}
	\newcommand{\pu}{\vec{u}}
	\newcommand{\pv}{\vec{v}}
	\newcommand{\pn}{\vec{n}}
	\newcommand{\pp}{\vec{p}}
	\newcommand{\pq}{\vec{q}}
	\newcommand{\pl}{\vec{l}}
	\newcommand{\vt}{\rm{vec}}
	\newcommand{\vx}{\mathbf{x}}
	\newcommand{\vy}{\mathbf{y}}
	\newcommand{\vu}{\mathbf{u}}
	\newcommand{\vv}{\mathbf{v}}
	\newcommand{\y}{\mathbf{y}}
	\newcommand{\vz}{\mathbf{z}}
	\newcommand{\T}{\top}
	
	\newtheorem{Thm}{Theorem}[section]
	\newtheorem{Def}[Thm]{Definition}
	\newtheorem{Ass}[Thm]{Assumption}
	\newtheorem{Lem}[Thm]{Lemma}
	\newtheorem{Prop}[Thm]{Proposition}
	\newtheorem{Cor}[Thm]{Corollary}
	\newtheorem{example}[Thm]{Example}
	\newtheorem{remark}[Thm]{Remark}
	
	\section{Introduction}
	
	
	A gain graph assigns an element of a mathematical group to each of its edges, and if a group element is assigned to an edge, then the inverse of that group element is always assigned to the inverse edge of that edge \cite{CDD21,HG12,Re16}.    {If such a mathematical group} consists of unit numbers of a number system,  {then} the gain graph is called a unit gain graph.   The real unit gain graph is called a signed graph 
	\cite{ABHMS19, Be14, Fa05, Ho05, HLP03}. {This was further extended to signed hypergraphs \cite{YYQ19}.}  Since 2012, 
		the study on  {complex unit gain graphs} has grown explosively {\cite{AABBd19,  HHD20,  LW20, Re12, WGF18}}. Then, started in 2022, 
		{quaternion} unit gain graphs and their associated spectral theories  were studied in \cite{BBCRS22, KTP23, ZL23}.  The study of  unit gain graphs  and their spectral properties forms an important part of spectral graph theory.
	
	 Dual numbers, dual quaternions, dual complex numbers, and  their applications have a long history.   It was British mathematician William Kingdon Clifford who introduced dual numbers in 1873 \cite{Cl73}.  Then German mathematician Eduard Study introduced dual angles in 1903 \cite{St03}.   These started the study and applications of dual numbers in kinematics, dynamics,  robotics, and brain science \cite{CLQY22, Da99, WYL12,ACVL17,WDW24}. Especially,
	 unit dual quaternion is an efficient  mathematical tool to describe the rigid body movement in the 3D space \cite{WYL12}.  A unit dual quaternion  serves as both a specification of the configuration (position and orientation) of a rigid body and a transformation taking the coordinates of a point from one frame to another via rotation and translation.  	It has wide applications in  robot control \cite{ACVL17},  formation control \cite{WYL12},  hand-eye  calibration  \cite{CLQY22, Da99}, and  simultaneous localization and  mapping  \cite{CKJC16}, etc.
	{Recently, dual complex numbers found application in brain science \cite{WDW24}. Very recently, dual quaternion unit gain graphs were applied to multi-agent formation control \cite{QC24}.}  This stimulates us to study dual quaternion unit gain graphs further.
	
 In this paper, we study dual quaternion unit gain graphs as well as dual complex unit gain graphs.   We combine them in a unified frame of dual unit gain graphs.  Dual real unit gain graphs are nothing else, but signed graphs.  We adopt such a unified way to avoid unnecessary repetitions.

	In the next section, we review  some preliminary  knowledge of  gain graphs,  dual elements and dual  matrices.
	We study some basic properties of  dual unit gain graphs   in Section 3.
 By using the dual  cosine functions,  we  establish  the closed form of eigenvalues of adjacency and Laplacian matrices  of dual complex  and  quaternion unit  gain cycles. 	In Section 4, we discuss the properties of eigenvalues of adjacency matrices of dual   unit gain graphs.  We establish the interlacing theorem  and show that the spectral radius of a dual  unit gain graph is always not greater than the spectral radius of the underlying graph, and these two radii are equal if and only if the dual   unit gain graph is balanced.  We  show the coefficient theorem holds for  dual unit gain graphs in Section 5.
	 Similar results hold for the spectral radius of the Laplacian matrix of the  dual unit gain graph  too.  We study these  in Section 6.

 {\section{Gain Graphs, Dual Elements and Dual Matrices}
		
	 The field of real numbers, the field of complex numbers and the ring of quaternions are denoted  by  $\rset, \cset$ and $\qset$, respectively. Following \cite{QC24}, we use $\gset$ to represent them.   Thus, $\gset$ may be $\rset$ or $\cset$ or $\qset$.
			\medskip
		
		\subsection{Gain graphs}}
	
	Suppose $G=(V,E)$ is a  graph, where $V:=\{v_1,v_2,\dots,v_n\}$ and   $E:=\{e_{ij}\}$.
	Define $n:=|V|$ and $m:=|E|$. The degree of a vertex $v_j$ is denoted by $d_j=\text{deg}(v_j)$ and the maximum degree is $\Delta=\max d_j$.
	An oriented edge from $v_i$ to $v_j$ is denoted by $e_{ij}$.
	The set of oriented edges, denoted by $\vec{E}(G)$, contains two copies of each edge with opposite directions.
	{Then $\left|\vec{E}\right| = 2m$. Even though $e_{ij}$ stands for   an edge and an oriented edge simultaneously, it will always be clear in the content.
		By graph theory,  $G$ is bidirectional.}
	
	A {\sl gain graph} is a triple $\Phi=(G, \Omega,\varphi)$ consisting of an {\sl underlying graph} $G=(V,E)$, the {\sl gain group} $\Omega$ and the {\sl gain function} $\varphi: \vec{E}(G)\rightarrow \Omega$ such that $\varphi(e_{ij}) = \varphi^{-1}(e_{ji})$.
	The gain group may be  $\{ 1, -1 \}$, or $\cgroup = \{ p \in \cset |  p \not = 0 \}$, or $\ucset = \{ p \in \cset | | p | = 1 \}$, etc.  If the gain group $\Omega$ consists of unit elements, then such a gain graph is called a {\bf unit gain graph}.
	If there is no  {confusion},  {then} we   simply denote the gain graph as $\Phi=(G, \varphi)$.

	A {\sl switching function} is a function $\zeta: V\rightarrow \Omega$ that switches the  $\Omega$-gain graph $\Phi=(G,\varphi)$ to $\Phi^{\zeta}=(G,\varphi^\zeta)$, where
	\[\varphi^{\zeta}(e_{ij}) = \zeta^{-1}(v_i)\varphi(e_{ij})\zeta(v_j).\]
	In this case, $\Phi$ and $\Phi^{\zeta}$ are switching equivalent, denoted by $\Phi\sim\Phi^{\zeta}$.
	Further, denote  the switching class of $\varphi$ as $[\Phi]$, which is the set of  gain graphs switching equivalent to $\Phi$.
	The gain of a walk $W=v_1e_{12}v_2e_{23}v_3\cdots v_{k-1}e_{k-1,k}v_k$ is
	\begin{equation}\label{equ:gain_walk}
		\varphi(W)=\varphi(e_{12})\varphi(e_{23})\cdots\varphi(e_{k-1,k}).
	\end{equation}
	A walk $W$ is {\sl neutral} if $\varphi(W)=1_{\Omega}$, where $1_{\Omega}$ is the identity of $\Omega$.
	An edge set $S\subseteq E$ is {\sl balanced} if every {cycle} $C\subseteq S$ is neutral. A subgraph is balanced if  its edge set is balanced.
	
	{Let $\Phi=(G,\Omega, \varphi)$ be a gain graph {and  {$A(\Phi) = (a_{ij}(\Phi))$} and  $L(\Phi)$ be the adjacency and Laplacian matrices of $\Phi$, respectively. Define 
			$A(\Phi)$ and $L(\Phi)$}
		via the gain function  {$\varphi(e)$} as follows,
		\begin{equation}\label{mat:Adj_phi}
			a_{ij}{(\Phi)} = \left\{
			\begin{array}{cl}
				\varphi(e_{ij}),  &  \text{if } e_{ij}\in {E(\Phi)},\\
				0, & \text{otherwise},
			\end{array}
			\right.  \ \text{ and } L(\Phi) =   D(\Phi)- A(\Phi).
		\end{equation}
		Here,  $\varphi(e_{ij}) \in {\Omega}$, $\varphi(e_{ij})=\varphi(e_{ji})^{-1}=\overline{\varphi(e_{ji})}$,  and $D(\Phi)$ is an $n\times n$ diagonal matrix with each diagonal element being the degree of the corresponding vertex {in its underlying graph
			$G$}.}

	Denote by $-\Phi$ the  {gain graph  $(G, \Omega, -\varphi)$} obtained by replacing the gain of each edge with its opposite.
	Clearly, $A(-\Phi) = -A(\Phi)$. Furthermore, $\Phi$ is {\sl antibalanced} if and only if $-\Phi$ is balanced.

	{A} {\sl potential function} is a function $\theta: V\rightarrow \Omega$ such that  for every $e_{ij}\in {E(\Phi)}$,
	\begin{equation}\label{equ:potential}
		\varphi(e_{ij}) = \theta(v_i)^{-1}\theta(v_j).
	\end{equation}
	We write $(G,1_{\Omega})$ as the $\Omega$-gain graph with all neutral edges.
	Note {the potential function} $\theta$ is not unique since for {any 
		$q\in\Omega$},  $\tilde\theta(v_i) = {\frac{q}{|q|}}\theta(v_i)$ for all $v_i\in V$ is also a potential function of $\Phi$.

	The following {result can be deduced from \cite{Za89}.  Also see \cite{BBCRS22}.}

	\begin{Lem} \label{lem:balance}
		Let {$\Phi=(G,\Omega, \varphi)$ be a gain} 
		graph. Then the following are equivalent:
		\begin{itemize}
			\item[(i)] $\Phi$ is balanced.
			\item[(ii)] $\Phi\sim(G,1_{\Omega})$.
			\item[(iii)] $\varphi$ has a potential function.
		\end{itemize}
	\end{Lem}

	\subsection{Dual elements}

			Denote  $\epsilon$ as the infinitesimal unit, satisfying $\epsilon\neq 0$ and  $\epsilon^2 = 0$.   The symbol $\epsilon$ is commutative with numbers in $\gset$.
			If $a_s,a_d\in\rset$, then $a = a_s + a_d\epsilon \in \drset$ is a dual   number.
			Similarly, 	if $a_s,a_d\in\cset$, then $a = a_s + a_d\epsilon \in \dcset$ is a dual complex number; if $a_s,a_d\in\qset$, then $a = a_s + a_d\epsilon \in \dqset$ is a dual quaternion number.  Here,  $\drset, \dcset$ and $\dqset$   denote the ring of dual numbers, dual complex numbers and dual quaternions, respectively.
			We use $\dgset$ to represent them in general.   We call an element in $\dgset$ a dual element.
			

			A {\bf dual element} $a = a_s + a_d\epsilon \in \dgset$ has standard part $a_s \in \gset$ and dual part $a_d \in \gset$.    The {\bf conjugate} of $a$ is defined as $a^* = a_s^* + a_d^*\epsilon$, where $a_s^*$ and $a_d^*$ is the conjugates of numbers $a_s$ and $a_d$, respectively.   Note
			that if $a\in\drset$, then $a^*=a$.      If $a_s \not = 0$, then we say that $a$ is {\bf appreciable}.   Otherwise, we say that $a$ is {\bf infinitesimal}.
			The real part of $a$ is defined by $Re(a) = Re(a_s)+Re(a_d)\epsilon$, where $Re(a_s)$ and $Re(a_d)$ are the real parts of the  numbers $a_s$ and $a_d$, respectively.
			The standard part   $a_s$ and the dual part $a_d$  of $a$ are denoted by $St(a)$  and $Du(a)$, respectively.

			Suppose we have two dual elements $a = a_s + a_d\epsilon$ and $b = b_s + b_d\epsilon$.   Then their sum is $a+b = (a_s+b_s) + (a_d+b_d)\epsilon$, and their product is $ab = a_sb_s + (a_sb_d+a_db_s)\epsilon$.
			In this way, $\dgset$ is a ring.   In particular, $\drset$ and $\dcset$ are two commutative rings, while $\dqset$ is a noncommutative ring. 		
			
			If both $a_s$ and $a_d$ are real numbers, then $a = a_s + a_d\epsilon$ is called a dual number.  Suppose we have two dual numbers $a = a_s + a_d\epsilon$ and $b = b_s + b_d\epsilon$.   By \cite{QLY22}, if
			$a_s > b_s$, or $a_s = b_s$ and $a_d > b_d$, then we say $a > b$.   Then this defines positive, nonnegative dual numbers, etc.

	For a dual element $p = p_s + p_d\epsilon\in\dgset$, its squared norm is a nonnegative dual number defined by
	$$|p|^2 = p^*p = p_s^*p_s + (p_s^*p_d+p_d^*p_s)\epsilon,$$
	and its magnitude is defined as a nonnegative dual number as follows
	\begin{equation} \label{e1}
		|p| := \left\{ \begin{aligned} |p_s| + {(p_s^*p_d+p_d^*p_s) \over 2|p_s|}\epsilon, & \ {\rm if}\  p_s \not = 0, \\
			|p_d|\epsilon, &  \ {\rm otherwise}.
		\end{aligned} \right.
	\end{equation}
	The dual element $p$ is invertible if and only if $p_s\neq 0$. Furthermore, if $p$ is invertible, then $p^{-1}=\frac{p^*}{|p|^2}$.
	
	We study unit dual complex and quaternion numbers in a unified frame as  {\bf unit dual elements}.  We denote this group by $\udset$.
		Then
		$$\udset = \{ p \in \dgset | | p | = 1 \}.$$  A unit dual number is either $1$ or $-1$.
		A dual complex number $p = p_s + p_d\epsilon$ is a {\bf unit dual complex number} if and only if $|p_s|=1$ and $p_sp_d^* + p_dp_s^* = 0$.
		We denote this group by $\udcset$.
		Then
		$$\udcset = \{ p \in \dcset | | p | = 1 \}.$$
		A dual quaternion number $p = p_s + p_d\epsilon$ is a {\bf unit dual quaternion number} if and only if $|p_s|=1$ and $p_sp_d^* + p_dp_s^* = 0$.
		We denote this group by $\udqset$.
		Then
		$$\udqset = \{ p \in \dqset | | p | = 1 \}.$$
		A unit dual element is always invertible {and} its inverse is its conjugate.   	The product of two unit dual  elements is still a unit dual element.  Hence, the set of unit dual  elements forms a group by multiplication.

		Given a graph $G$ and a certain  gain function $\varphi: \vec{E}(G)\rightarrow \udset$.  We call $\Phi=(G,\udset, \varphi)$ a {\bf dual unit gain graph}.
		Thus, $\Phi$ may be a dual quaternion unit gain graph, a dual complex unit gain graph, or a signed  graph if $\udset=\udqset, \udcset, \{\pm 1\}$, respectively.

	%
	%
	%
	{\begin{Lem}\label{prelim_dc}
			Let $\dc q_1$ and  $\dc q_2 \in \dgset$ be two dual elements. Then
			\begin{itemize}
				\item[(i)] $Re(\dc q_1)\le |\dc q_1|$ and the equality holds if and only if $\dc q_1$ is a nonnegative dual number.
				\item[(ii)] $Re(\dc q_1) =Re({\dc q^*_1})$ and $Re(\dc q_1\dc q_2) =Re(\dc q_2\dc q_1)$.
				\item[(iii)] $|\dc q_1\dc q_2| = |\dc q_1| |\dc q_2|$.
				\item[(iv)] $|\dc q_1+\dc q_2| \le |\dc q_1| +|\dc q_2|$.
			\end{itemize}
		\end{Lem}
	}

	Two dual elements   $p$ and $q$ are {\bf similar}  if there is an appreciable dual element $u$ such that $p=u^{-1}qu$. We denote $p\sim q$ and denote by $[q]$ the equivalence class containing $q$.
	In Lemma 2.1 of  \cite{Zh97}, it was shown   that any quaternion number is similar to a complex number and the real part and the absolute value of the imaginary part of there two numbers are the same. In the following  proposition, we generalize this result to dual quaternion numbers.
	
	\begin{Prop}\label{prop:dq2dc}
		Suppose $q=q_s+q_d\epsilon\in\dqset$. Then there exists $a\in\dcset$  belonging  to $[q]$. In other words, for any  dual quaternion number $q\in\dqset$, there is  $u\in \udqset$ and $a\in\dcset$ such that $a=u^*qu.$ Furthermore, there is $Re(q)=Re(a)$ and $|Im(q)|=|Im(a)|$.
	\end{Prop}
	\begin{proof}
		Denote $q_s = q_0+ q_{11}i+q_{12}j+q_{13}k$, $\vq_1=[q_{11},q_{12},q_{13}]^\top$, $q_d = q_2+ q_{31}i+q_{32}j+q_{33}k$,  $\vq_3=[q_{31},q_{32},q_{33}]^\top$, and $a=u^*qu$.  Consider the following two cases.

		Suppose that $q_s\in\rset$.  Let  $x=q_{31}+\|\vq_3\|{-q_{33}j+q_{32}k}$, $u_s=\frac{x}{|x|}$,   $u_d=0$.
		Then we can verify that $u_s^*q_du_s=q_2+\|\vq_3\|i$.  Thus,
		\begin{eqnarray*}
			a = u^*qu&=&u_s^*q_su_s + (u_s^*q_su_d+u_d^*q_su_s+u_s^*q_du_s)\epsilon\\
			&=&u_s^*u_s q_s + (u_s^*u_d+u_d^*u_s)q_s\epsilon+u_s^*q_du_s\epsilon\\
			&=&  q_s +  u_s^*q_du_s\epsilon\in\dcset.
		\end{eqnarray*}
		Furthermore, there is $Re(q)=Re(a)=q_0+q_2\epsilon$ and $|Im(q)|=|Im(q_d)|\epsilon =\|\vq_3\|\epsilon= |Im(a)|$.
		
		Suppose that $q_s\in\qset$. Let $x = q_{11}+\|\vq_1\|{-q_{13}j+q_{12}k}$ and $u=u_s+u_d\epsilon\in\udqset$ be a unit dual quaternion number with $u_s=\frac{x}{|x|}$.
		It follows from \cite{Zh97} that $a_s=x^{-1}q_sx =u_s^*q_su_s=q_0+\|\vq_1\|i \in\cset$.
		Thus, $Re(p_s)=Re(a_s)$ and $|Im(p_s)|=|Im(a_s)|$.
		Next, we focus on the dual part of $u^*qu$. Let $u_d=\frac12 u_st$, and $t=t_{11}i+t_{12}j+t_{13}k$ be  an infinitesimal quaternion.  Then $u\in\udqset$ for any  infinitesimal quaternion number $t$. By direct   computation, we have
		\begin{eqnarray*}
			a_d= Du(u^*qu) &=& u_d^*q_su_s + u_s^*q_su_d + u_s^*q_du_s\\
			&=&  u_d^*u_su_s^*q_su_s + u_s^*q_su_su_s^*u_d + u_s^*q_du_s\\
			&=&  u_d^*u_sa_s + a_su_s^*u_d + u_s^*q_du_s\\
			&=&   \frac12(-ta_s + a_s t) + u_s^*q_du_s\\
			&=& \|\vq_1\|(t_{12}k-t_{13}j) +  u_s^*q_du_s,
		\end{eqnarray*}
		where the last inequality follows from the definitions.
		Let $r=u_s^*q_du_s=r_0+r_{11}i+r_{12}j+r_{13}k\in\qset$, $t_{11}=0$, $t_{12}=-r_{13}\|\vq_1\|^{-1}$ and  $t_{13}= r_{12}\|\vq_1\|^{-1}$. Then there is $a_d = Du(u^*qu)=r_0+r_{11}i\in\cset$, where $r_0=q_2$, $r_{11} =  \vq_1^\top\vq_3 \|\vq_1\|^{-1}$.
		Furthermore, there is $|Im(a)|=|\|\vq_1\|+r_{11}\epsilon|=\|\vq_1\|+r_{11}\epsilon$ and $|Im(q)|=|Im(q_s)+Im(q_d)\epsilon|=\|\vq_1\|+\vq_1^\top\vq_3 \|\vq_1\|^{-1}\epsilon$.
		Thus, $|Im(a)|=|Im(q)|$.
		
		This completes the proof. 	
	\end{proof}

	{Let $a=a_s+a_d \epsilon \in\dcset$. Then  the exponential function of $a$ is
		\begin{equation}\label{equ:dualexp}
			e^{a} = e^{a_s} +  a_d e^{a_s}\epsilon.
		\end{equation}
		If $a$ is appreciable, its logarithm function   is
		\begin{equation}\label{equ:duallog}
			\log(a) = \log(a_s)+a_s^{-1}a_d\epsilon.
		\end{equation}
		Let $\theta=\theta_s+\theta_d \epsilon \in\drset$ be a dual angle. Then the cosine function of $\theta$ is
		\begin{equation}\label{equ:dualcos}
			\cos(\theta) = \cos(\theta_s)-\theta_d\sin(\theta_s)\epsilon.
		\end{equation}
		
		\begin{Prop}\label{prop:dcnroot}
			Let $a=a_s+a_d \epsilon \in\dcset$. Then the following results hold.
			
			(i) For any unit dual complex number $a\in\udcset$, there is a dual angle $\theta=\theta_s+\theta_d\epsilon\in \drset$ such that $a=e^{i\theta}$. Here, $\theta=-i\log(a) = -i(\log(a_s)+a_s^*a_d\epsilon)$.
			
			(ii) For any positive integer $n$, there is 	 $a^{\frac1n} = e^{i\frac{\theta+2\pi j}{n}}$ for any $j\in \mathbb Z.$
			
			(iii) $\cos(\theta) = \frac12 \left(e^{i\theta}+e^{-i\theta}\right).$
		\end{Prop}
		\begin{proof}
			(i) Since $|a_s|=1$, there exists $\theta_s=-i\log(a_s)$ such that $a_s=e^{i\theta_s}$. Furthermore, let $\theta_d=-ia_d a_s^*$. We can verify that  $a_d=i\theta_d e^{i\theta_s}$. It follows from   $a_d a_s^*+a_s a_d^*=0$ that  $\theta_d\in\mathbb R$.
			
			(ii)  Let  $z_0=e^{i\frac{\theta}{n}}$. We can verify that $\left(z_0\right)^n=a$. Suppose that $z=cz_0$ and $z^n=a$. Then we can verify that $c^n=1$. Thus, there is $z=e^{i\frac{\theta+2\pi j}{n}}$ for any $j\in\mathbb Z$.
			
			(iii) By direct computation, there is
			\begin{eqnarray*}
				e^{i\theta}+e^{-i\theta} &=&  e^{i\theta_s} +  i\theta_d e^{i\theta_s}\epsilon + e^{-i\theta_s} -  i\theta_d e^{-i\theta_s}\epsilon \\
				&=& 2\cos(\theta_s) + i\theta_d(e^{i\theta_s}-e^{-i\theta_s})\\
				&=& 2\cos(\theta_s) -\theta_d\sin(\theta_s)\\
				&=& 2\cos(\theta).
			\end{eqnarray*}
			
			This completes the proof.
		\end{proof}
	}
	
	An $n$-dimensional {\bf dual    vector} is denoted by $\vx = (x_1,\dots, x_n)^\top$, where $x_1,\dots, x_n\in\dgset$ are dual elements.   We may denote $\vx = \vx_s + \vx_d\epsilon$, where $\vx_s$ and $\vx_d$ are two $n$-dimensional  vectors in $\gset^n$. The $2$-norm of $\vx$ is defined as
	$$\|\vx\|_2 = \left\{\begin{array}{ll}
		\sqrt{\sum_{i=1}^n |x_i|^2}, & \ \mathrm{if} \  \vx_s \not = \0,\\
		\|\vx_d\|_2\epsilon, & \ \mathrm{if} \  \vx_s  = \0.
	\end{array}\right.$$
	If $\vx_s \not = \0$, then we say that $\vx$ is appreciable.   
	Denote $\vx^* = (x_1^*,\dots, x_n^*)$ as the conjugate of $\vx$.  Let $\vy = (y_1,\dots, y_n)^\top$ be another $n$-dimensional dual  vector.   Define
	$$\vx^*\vy = \sum_{j=1}^n x_j^*y_j.$$
	If $\vx^*\vy = 0$, we say that $\vx$ and $\vy$ are {\bf orthogonal}.  Note that $\vx^*\vx = \|\vx\|_2^2$.
	Let $\vx_1,\dots, \vx_n\in\dgset^n$ be $n$ dual  vectors.   If $\vx_i^*\vx_j = 0$ for $i \not = j$ and $\vx_i^*\vx_j = 1$ for $i = j$, $i, j = 1,\dots, n$, then we say that $\vx_1,\dots, \vx_n$ form an {\bf orthonormal basis} of the $n$-dimensional dual  vector space.   

	\subsection{Dual   matrices}
	
	Assume
	that $A = A_s + A_d\epsilon$ and $B = B_s + B_d\epsilon$ are two   dual   matrices in $\dgset^{n\times n}$, where $n$ is a positive integer, $A_s, A_d, B_s$ and $B_d$ are four   matrices in $\gset^{n\times n}$.
	If $AB = BA = I$, where
	$I$ is the $n \times n$ identity matrix, then we say that $B$ is the {\bf inverse} of $A$ and denote that $B = A^{-1}$.  We have the following proposition.
	
	\begin{Prop} \label{p2.1}
		Suppose that $A = A_s + A_d\epsilon \in \dgset^{n\times n}$ and $B = B_s + B_d\epsilon \in \dgset^{n\times n}$ are two  dual  matrices.   Then the following four statements are equivalent.
		
		(a) $B = A^{-1}$;
		
		(b) $AB = I$;

		(c) $A_sB_s = I$ and $A_sB_d+A_dB_s=O$, where $O$ is the $n \times n$ zero matrix;
		
		
		(d) $B_s = A_s^{-1}$ and $B_d= -A_s^{-1}A_dA_s^{-1}$.
	\end{Prop}
	
	
	Given a dual  matrix $A\in\dgset^{n\times n}$, denote its conjugate transpose as $A^*$.   If $A^* = A^{-1}$, then $A$ is called a {\bf dual unitary matrix}.  If a dual number matrix is a dual unitary matrix, then we simply call it a {\bf dual orthogonal matrix}.
	Let $A\in \dgset^{n\times m}$. Then the null space generated by $A$ is
	\[\text{null}(A)= \{{\vx\in\dgset^n}\,|\, A^*{\vx}=0\},\]
	and the span of $A$ is
	\[\text{span}(A)=\{{\vx\in\dgset^n}\,|\, \exists{\vy} \in \dgset^m \text{ such that }{\vx} =A{\vy}\}.\]

	\begin{Prop}\label{Lem:DCspace}
		Let $U = [\vu_1,\dots, \vu_n]\in\dgset^{n\times n}$ be such that its columns form an orthonormal basis in $\dgset^n$. Then ${U^* U}=I$ and  the following results hold.
		
		(i) For  any vector ${\vx}\in  \dgset^n$, there exists   ${\vy}\in  \dgset^n$  such that {$\vx= U \vy$}.
		
		(ii) Suppose that $1<k<n$ and {$U = [U_1, U_2]$}, where ${U}_1\in\dcset^{n\times k}$ and $ {U}_2\in\dgset^{n\times (n-k)}$. Then  $\text{null}({U}_1) = \text{span}({U}_2)$.
	\end{Prop}
	\begin{proof}
		(i) This result  follows directly from  $\vy=U^*\vx$ and $UU^*=I$.
		
		
		(ii) By  ${U^*U}=I$, {we have} ${U_1^*U_2} = O_{k\times (n-k)}$, the $k \times (n-k)$ zero matrix.
		Let $\vx=U \vy = U_1 \vy_1 + U_2 \vy_2$, {where $\vy_1$ and $\vy_2$ have the same partition as the columns of  $U_1$ and $U_2$, respectively.}
		On one hand, if {$\vx \in \text{null}(U_1)$}, then {$\vy_1=U_1^*\vx= {\mathbf 0}$}. Thus  ${\vx}=U_2\vy_2 \in \text{span}({U}_2)$ and  $\text{null}({U}_1) \subseteq \text{span}({U}_2)$.
		On the other hand, if ${\vx} \in \text{span}({U}_2)$, then there exists $\vy_2\in\dgset^{n-k}$ such that  {$\vx=U_2 \vy_2$}. Thus,  $ {U_1^*\vx= {\mathbf 0}}$ and $\text{span}({U}_2)\subseteq\text{null}({U}_1)$.
		
		This completes the proof.
	\end{proof}

	Assume that $A\in\dgset^{n\times n}$ and $\lambda = \lambda_s + \lambda_d\epsilon\in\dgset$.	If
	\begin{equation} \label{en1}
		A\vx =\vx  \lambda,
	\end{equation}
	where $\vx$ is appreciable, i.e., $\vx_s \not = \0$, then $\lambda$ is called a {\bf right eigenvalue} of $A$, with an eigenvector $\vx$. Similarly, if
	\begin{equation*}
		A\vx = \lambda \vx,
	\end{equation*}
	where $\vx$ is appreciable, i.e., $\vx_s \not = \0$, then $\lambda$ is called a {\bf left eigenvalue} of $A$, with an eigenvector $\vx$.
	{If $\dgset$ is $\drset$ or $\dcset$, then the multiplication is commutative.    In these two cases, it is not {necessary} to distinguish right and left eigenvalues.  We just call them eigenvalues \cite{QC23}.
		
		A Hermitian matrix in $\gset$ means a symmetric matrix, or a complex Hermitian matrix, or a quaternion Hermitian matrix, depending upon $\gset = \rset$, or $\cset$, or $\qset$.
		Similarly,  a dual Hermitian matrix in $\dgset$ means a dual symmetric matrix, or a dual complex Hermitian matrix, or a dual quaternion Hermitian matrix, depending upon $\dgset = \drset$, or $\dcset$, or $\dqset$.  	
		By \cite{QC23}, a non-Hermitian dual number matrix may have no eigenvalue at all, or have
		infinitely many eigenvalues.
		However, a dual quaternion Hermitian matrix  has exactly $n$ dual number right eigenvalues  \cite{QL23}.
		As dual numbers are commutative with dual quaternions, they are also left eigenvalues.  Thus, we may simply call them {\bf eigenvalues} of $A$.  Note that $A$ may still have other left eigenvalues, which are not dual numbers.  See an example of a quaternion matrix in \cite{Zh97}.
		
		Following \cite{QL23}, we  may prove  that if $A\in\dgset^{n\times n}$ is Hermitian, then it has exactly $n$ dual number eigenvalues, with orthonormal
		eigenvectors. Furthermore, $A$ is positive semidefinite (definite) if and only if its eigenvalues are nonnegative (positive).}
	
	If there is an $n \times n$ invertible dual   matrix $P$ such that $A = P^{-1}BP$, then we say that $A$ and $B$ are {\bf similar}, and denote $A \sim B$.
	We have the following proposition.
	
	\begin{Prop} \label{p2.2}
		Suppose that $A$ and $B$ are two $n \times n$ dual   matrices, $A \sim B$, i.e., $A = P^{-1}BP$ for some $n \times n$ invertible dual  matrix $P$, and $\lambda$ is a dual element eigenvalue of $A$ with a dual element eigenvector $\vx$.   Then $\lambda$ is an eigenvalue of $B$ with an eigenvector $P\vx$.
	\end{Prop}
	
	Suppose that $A = A_s + A_d\epsilon$, $\lambda = \lambda_s + \lambda_d \epsilon$ and $\vx = \vx_s + \vx_d \epsilon$. Then
	(\ref{en1}) is equivalent to
	\begin{equation} \label{en2}
		A_s\vx_s = \lambda_s\vx_s,
	\end{equation}
	with $\vx_s \not = \0$, i.e., $\lambda_s$ is an eigenvalue of $A_s$ with an eigenvector $\vx_s$, and
	\begin{equation} \label{en3}
		(A_s-\lambda_sI)\vx_d - \lambda_d\vx_s = -A_d\vx_s.
	\end{equation}
	{Recently, several numerical methods for computing eigenvalues of dual quaternion Hermitian matrices arose.  These include a power method \cite{CQ23}, a bidiagonalization method \cite{DLWW24} and a Rayleigh quotient iteration method \cite{DWD24}.}

	%
	%
	%
	%

	\section{Dual   Unit Gain Graphs}
	
	Let $\Phi=(G,\udset, \varphi)$ be a   dual   unit gain graph. 
	Denote 	 the adjacency and Laplacian matrices of $\Phi$  defined by \eqref{mat:Adj_phi} as $A$ and $L$, respectively.
	By Lemma~\ref{lem:balance}, we have the following lemma for dual  unit gain graphs.  
	
	\begin{Lem} \label{Reasonable}
		Let $\Phi=(G,\udset, \varphi)$ be a   dual  unit gain graph. Then the following are equivalent:
		\begin{itemize}
			\item[(i)] $\Phi$ is balanced.
			\item[(ii)] $\Phi\sim(G,1_{\udset})$.
			\item[(iii)] $\varphi$ has a potential function.
		\end{itemize}
	\end{Lem}

	Let $\Phi=(G,\udset, \varphi)$ be a   dual  unit gain graph with $n$ vertices. Then {the adjacency matrix} $A(\Phi)$ and the Laplacian matrix $L(\Phi)$ are two $n \times n$ dual   Hermitian matrices.   Thus, each of $A(\Phi)$ and $L(\Phi)$ has $n$ eigenvalues.  The set of the $n$ eigenvalues of $A(\Phi)$ is called the {\bf spectrum} of $\Phi$, and denoted as $\sigma_A(\Phi)$, while		the set of the $n$ eigenvalues of $L(\Phi)$ is called the {\bf Laplacian spectrum} of $\Phi$, and denoted as $\sigma_L(\Phi)$.
	{We also denote the eigenvalue {set}s of the adjacency and Laplacian matrices of the underlying graph $G$ as $\sigma_A(G)$ and $\sigma_L(G)$, respectively.}
	We have the following theorem.
	
	\begin{Thm} \label{spectrum}
		Let $\Phi=(G,\udset, \varphi)$ be a   dual   unit gain graph with $n$ vertices.   If $\Phi$ is balanced, then   $\sigma_A(\Phi)=\sigma_A(G)$ and $\sigma_L(\Phi)=\sigma_L(G)$.
		Furthermore, the  spectrum
		$\sigma_A(\Phi)$ consists of $n$ real numbers, while the Laplacian spectrum $\sigma_L(\Phi)$ consists of one zero and $n-1$ positive numbers.
	\end{Thm}

	\begin{proof}
		{By Lemma \ref{Reasonable}, we have} $\Phi\sim \Phi^1:= (G,1_{\udset})$. 	Then  $\Phi^1$ is exactly the {underlying} graph $G$, and the adjacency and the Laplacian matrices of $\Phi^1$ {are the adjacency and  Laplacian matrices of $G$}, respectively.  It follows from $\Phi\sim \Phi^1$ that there exists a function $\zeta$ such that  for all $e_{ij}\in {E(\Phi)}$, there is
		\[\zeta^{-1}(v_i)\varphi(e_{ij})\zeta(v_j)=1.\]
		Thus, we have
		\[A(\Phi^1) = Z^{-1}A(\Phi)Z, \text{   where  }Z=\text{diag}(\zeta(v_1),\dots,\zeta(v_n)),\]
		and $L(\Phi^1) = {D(\Phi^1)}-A(\Phi^1)  = Z^{-1}({D(\Phi)}-A(\Phi))Z =Z^{-1}L(\Phi)Z$.
		In other words,  $A(\Phi^1)$ and $L(\Phi^1)$ are similar with $A(\Phi)$ and $L(\Phi)$, respectively.   {Now, the conclusions of this theorem follow from the spectral properties of adjacency and Laplacian matrices of ordinary graphs.}
	\end{proof}

	If $G$ is a tree, then $\Phi$ is balanced, and the eigenvalues of the adjacency and the Laplacian matrices of $\Phi$ are the same with {those} of the adjacency and  Laplacian matrices of $G$, respectively. For instance, a path is a special tree. Let $P_n$ be the path   on $n$ vertices and $\Phi=(P_n,\udcset,\varphi$).
	Then the eigenvalues of $A(\Phi)$ and $L(\Phi)$ can be calculated as
	\[\sigma_A(\Phi) =  \left\{2\cos\left(\frac{\pi j}{n+1}\right):\ j\in\{1,\dots,n\}\right\},\]
	and
	\[\sigma_L(\Phi) =  \left\{2-2\cos\left(\frac{\pi j}{n}\right):\ j\in\{0,\dots,n-1\}\right\},\]
	respectively.

	If $\Phi$ is not balanced, then the eigenvalues of $A(\Phi)$ may not be real numbers.  The following example illustrates this.

	\begin{figure}\label{ex:dcugg}
		\begin{subfigure}[b]{0.3\textwidth}
			\centering
			\includegraphics[width=0.85\linewidth]{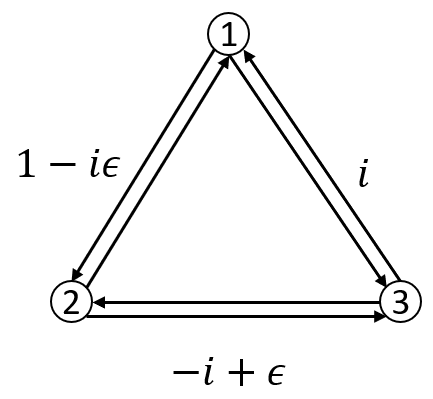}
			\caption{$\udcset$-gain graph $\Phi_1$}
			\label{fig:balanced}
		\end{subfigure}
		\hfill
		\begin{subfigure}[b]{0.3\textwidth}
			\centering
			\includegraphics[width=0.85\linewidth]{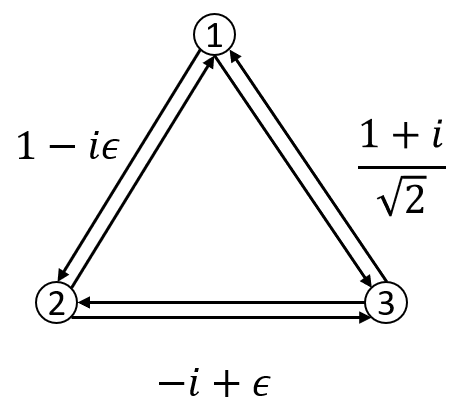}
			\caption{$\udcset$-gain graph $\Phi_2$}
			\label{fig:balanced1}
		\end{subfigure}
		\hfill
		\begin{subfigure}[b]{0.3\textwidth}
			\centering
			\includegraphics[width=0.85\linewidth]{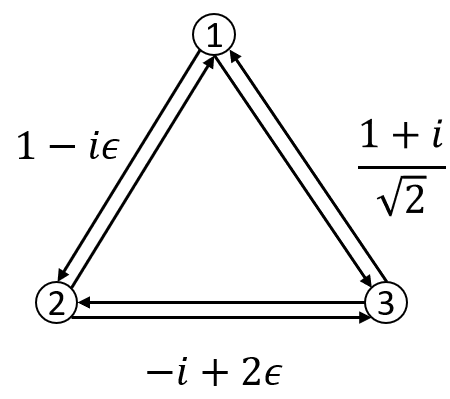}
			\caption{$\udcset$-gain graph $\Phi_3$}
			\label{fig:balanced1}
		\end{subfigure}
		\caption{Three $\udcset$-gain graphs}
	\end{figure}
	
	\begin{example}
		Consider {three} $\udcset$-gain  cycles in Figure 1.
		Their adjacency matrices are given as follows.
		\[A(\Phi_1) = \begin{bmatrix}
			0 & 1-i\epsilon & -i \\
			1+i\epsilon & 0 & -i+\epsilon \\
			i & i+\epsilon & 0
		\end{bmatrix},
		\quad
		A(\Phi_2) = \begin{bmatrix}
			0 & 1-i\epsilon & {\frac{1-i}{\sqrt{2}}}\\
			1+i\epsilon & 0 & -i+\epsilon \\
			\frac{1+i}{\sqrt{2}}   & i+\epsilon & 0
		\end{bmatrix},
		\]
		and
		\[A(\Phi_3) = \begin{bmatrix}
			0 & 1-i\epsilon & {\frac{1-i}{\sqrt{2}}}\\
			1+i\epsilon & 0 & -i+2\epsilon \\
			\frac{1+i}{\sqrt{2}}   & i+2\epsilon & 0
		\end{bmatrix}, \]
		respectively.
		The $\udcset$-gain graph $\Phi_1$ is balanced, and the eigenvalues of its adjacency {matrix} are
		\[\sigma_A(\Phi_1) = \{2,\quad  -1,\quad -1\},\]
		which are the same with the eigenvalues of the adjacency {matrix} of its underlying graph.
		{Both the $\udcset$-gain graphs $\Phi_2$ and  $\Phi_3$ are unbalanced.} The eigenvalues  of $A(\Phi_2)$ are all real numbers, i.e.,
		\[\sigma_A(\Phi_2) = \{1.9319,\quad -0.5176,\quad -1.4142\},\]
		while the eigenvalues  of $A(\Phi_3)$ 	are all dual numbers, i.e.,
		\[\sigma_A(\Phi_3) =\{1.9319+{0.1725}\epsilon,\quad -0.5176{-0.6440}\epsilon,\quad -1.4142{+0.4714}\epsilon\}.\]
		The standard parts of $\sigma_A(\Phi_3)$ are the same with  $\sigma_A(\Phi_2)$.
	\end{example}
	
	The closed form of the adjacency and Laplacian eigenvalues of a cycle $\ucset$-gain graph is presented in Theorem 5.1 in \cite{Re12}.  Let $C_n$ be the cycle  on $n$ vertices   and $\Phi=(C_n,\ucset,\varphi$) and  $\varphi(C_n)=e^{i\theta}\in \ucset$.
	Then  the eigenvalues of $A(\Phi)$ and $L(\Phi)$ can be calculated as
	\begin{equation}\label{equ:adj_eig_cycle}
		\sigma_A(\Phi) =  \left\{2\cos\left(\frac{\theta+2\pi j}{n}\right):\ j\in\{0,\dots,n-1\}\right\},
	\end{equation}
	and
	\begin{equation}\label{equ:Lap_eig_cycle}
		\sigma_L(\Phi) =  \left\{2-2\cos\left(\frac{\theta+2\pi j}{n}\right):\ j\in\{0,\dots,n-1\}\right\},\end{equation}
	respectively.
	{Very recently, \cite{BBCRS22} studied the  unit quaternion cycles  $\Phi=(C_n, \uqset,\varphi$) with gain $q =\varphi(C_n) \in \uqset$. They showed that $\Phi$ is similar with a unit complex cycle  $\Phi^{\zeta}$.  Furthermore, denote $Re(q)+|Im(q)|i=e^{i\theta}$. Then  \eqref{equ:adj_eig_cycle} and \eqref{equ:Lap_eig_cycle}  are also eigenvalues of $A(\Phi)$ and $L(\Phi)$ of     unit quaternion cycles, respectively.
		By using the dual element version of    exponential,   logarithm, and    cosine functions,  we extend this result to dual unit complex  and dual quaternion gain cycles.}

	\begin{Thm}
		Let $C_n:=v_1e_{12}v_2\cdots v_ne_{n1}v_1$ be the cycle   on $n$ vertices   and $\Phi=(C_n,\udset,\varphi$). Suppose $q = \varphi(C_n)\in\udset$.  Then the following results hold.
		
		(i) There exists a switching function $\zeta:V(\Phi)\rightarrow \udset$ such that
		$\varphi^{\zeta}(e_{i,i+1}) =1$ for $i=1,\dots,n-1$ and $\varphi^{\zeta}(e_{n,1})=q$.
		
		{(ii)  Suppose $\udset=\udcset$.  Then there exists    a dual angle $\theta=\log(q) \in\drset$ such that the eigenvalues of $A(\Phi)$ and $L(\Phi)$ can be calculated by
			\eqref{equ:adj_eig_cycle} and \eqref{equ:Lap_eig_cycle},
			respectively. Here,  the dual logarithm   and the dual cosine functions are defined by \eqref{equ:duallog} and \eqref{equ:dualcos}, respectively.

			(iii)   Suppose $\udset=\udqset$.  Then there exists a dual complex number $a\in\udcset$ and another switching function $\zeta_1:V(\Phi)\rightarrow \udset$ such that
			$\varphi^{\zeta_1}(e_{i,i+1}) =1$ for $i=1,\dots,n-1$ and $\varphi^{\zeta_1}(e_{n,1})=a$.
			Let  $\theta=\log(a) \in\drset$. Then   the eigenvalues of $A(\Phi)$ and $L(\Phi)$ can be calculated by  \eqref{equ:adj_eig_cycle} and \eqref{equ:Lap_eig_cycle}, 		respectively.
		}
		
		%
	\end{Thm}
	
	\begin{proof}
		(i) Define the switching function $\zeta:V(\Phi)\rightarrow \udset$ as follows:
		\[\zeta(v_i)=\left\{\begin{array}{ll}
			1, & i=1,\\
			\left(\Pi_{j=1}^{i-1} \varphi(e_{j,j+1})\right)^{-1},& i=2,\dots,n.
		\end{array}\right.\]
		Then  there is
		\[\varphi^{\zeta}(e_{12}) = \zeta(v_1)^{-1}\varphi(e_{12})\zeta(v_2) =\varphi(e_{12})\varphi(e_{12})^{-1}=1. \]
		For $i=2,\dots,n-1$, there is
		\begin{eqnarray*}
			\varphi^{\zeta}(e_{i,i+1}) &=& \zeta(v_i)^{-1}\varphi(e_{i,i+1})\zeta(v_{i+1})\\
			&=&\Pi_{j=1}^{i-1} \varphi(e_{j,j+1})\varphi(e_{i,i+1})\left(\Pi_{j=1}^{i} \varphi(e_{j,j+1})\right)^{-1}\\
			&=&1.
		\end{eqnarray*}
		For $i=n$, there is
		\begin{eqnarray*}
			\varphi^{\zeta}(e_{n,1}) &=& \zeta(v_n)^{-1}\varphi(e_{n,1})\zeta(v_{1})
			=\Pi_{j=1}^{n-1} \varphi(e_{j,j+1})\varphi(e_{n,1})=\varphi(C_n).
		\end{eqnarray*}
		
		{(ii) The proof is a slight modification of   Theorem 6.1 in \cite{Re12}. Let
			\[P = \begin{pmatrix}
				0& 1 & 0 &     \cdots & 0  \\
				0 &0& 1 &  \cdots & 0\\
				\vdots & \vdots  & \ddots & \ddots & 0\\
				0 & 0 & 0 & \cdots & 1\\
				p&   0 & 0 & \cdots & 0
			\end{pmatrix},\]
			and  $\vx=(x_i)\in\dcset^n$ be an eigenvector of $P$ corresponding to the eigenvalue $\lambda$. Then there is $x_{i+1}=\lambda x_i$ for $i=1,\dots,n-1$ and $px_1 = \lambda x_n$. Thus, there is $(p-\lambda^n)x_1=0$.
			$x_1$ must be an appreciable dual number since $\vx$ is appreciable and  $x_i=\lambda^{i-1}x_1$ for $i=2,\dots,n$.
			Thus, $p=\lambda^n$.  It follows from Proposition \ref{prop:dcnroot} that $\lambda=e^{i\frac{\theta+2\pi j}{n}}$ for $j=0,\dots,n-1$.
			
			Furthermore, we have $PP^*=I$ and   $A=P+P^*=P+P^{-1}$.
			Thus, the eigenvalue  of $A$ equal to $\lambda+\lambda^{-1}$.
			By Proposition~\ref{prop:dcnroot}, there is $\lambda+\lambda^{-1}=2\cos\left(\frac{\theta+2\pi j}{n}\right)$.
			This derives the closed form of $\sigma_A(\Phi)$.
			The closed form of $\sigma_L(\Phi)$ follows directly from  $L(\Phi)=2-A(\Phi)$.
			
			(iii) By Proposition~\ref{prop:dq2dc}, there exists $u\in\udqset$ such that $a=u^*qu\in\udcset$.
			Let $\zeta_1(v_i)=\zeta(v_i)u$ for $i=1,\dots,n$. Then there is $\zeta_1:V(\Phi)\rightarrow \udset$ such that
			$\varphi^{\zeta_1}(e_{i,i+1}) =1$ for $i=1,\dots,n-1$ and $\varphi^{\zeta_1}(e_{n,1})=a$.
			The result of this item follows directly from that of item (ii).
		}		
		
		This completes the proof.
	\end{proof}

	\section{Eigenvalues of Adjacency Matrices}
	
	
	The switching class has a unique adjacency spectrum.

	\begin{Prop}\label{lem:similar}
		Let $\Phi_1=(G,\varphi_1)$ and $\Phi_2=(G,\varphi_2)$ {be both} $\udset$-gain graphs. If $\Phi_1\sim \Phi_2$, then {$\Phi_1$ and $\Phi_2$} have the same {spectrum}. That is, $\sigma_A(\Phi_1) = \sigma_A(\Phi_2)$.
	\end{Prop}

	Let {$\Phi=(G,\varphi)$} be a $\udset$-gain graph, and  $S$ be a subset of $V$.
	Denote $\Phi[S]$ as the induced subgraph of $\Phi$ with vertex set $S$, and $\Phi-S$ as $\Phi[V\setminus S]$, respectively.
	Both the adjacency matrices  $A(\Phi)$ and  $A(\Phi[S])$ are Hermitian matrices.
	{As stated in Section 2,}  an {$n\times n$}  dual Hermitian  matrix has  exactly $n$ eigenvalues, which are dual numbers.
	The following theorem shows the {eigenvalues} of $A(\Phi)$ interlace with {those} of  $A(\Phi[S])$.
	
	\begin{Thm}(Interlacing Theorem) \label{Thm:interlacing_adj} Let $\Phi=(G,\varphi)$ be a $\udset$-gain graph with $n$ vertices and $S$ be a subset of $V$ with $k$ vertices. 
		Denote the eigenvalues of $A(\Phi)$ and  $A(\Phi[S])$ by
		\[{\dc \lambda_1\ge \dc  \lambda_2\cdots \ge \dc \lambda_n  \quad \text{and} \quad \dc  \mu_1\ge \dc \mu_2\ge \cdots \ge \dc  \mu_k,}\]
		respectively.
		Then the following inequalities hold:
		\begin{equation}\label{equ:interlacing}
			\dc \lambda_i\ge \dc  \mu_i\ge \dc \lambda_{n+i-k},\quad \forall\, 1\le i\le k.
		\end{equation}
	\end{Thm}
	
	\begin{proof}
		Suppose {some basis} eigenvectors of  $A(\Phi)$ and  $A(\Phi[S])$ are $\{\vdc x_1,\dots,\vdc x_n\}$ and  $\{\vdc y_1,\dots,\vdc y_k\}$, respectively. Without loss of generality,  assume
		\[A(\Phi) =  \left[\begin{array}{cc}
			A(\Phi[S]) & \dc X^* \\ \dc X & \dc Z
		\end{array}\right].\]
		{For $1 \le i \le k$, d}efine the following vector spaces
		\[\dc U=\text{span}(\vdc x_i,\dots,\vdc x_n), \  \dc V=\text{span}(\vdc y_1,\dots,\vdc y_i),\  \dc W=\left\{\left(\begin{array}{c}
			{\vy}\\ 0
		\end{array}\right)\in\dcset^n,  {\vy} \in \dc V\right\}.\]
		{By Theorem 4.4 in \cite{LQY23} and {Proposition}~\ref{Lem:DCspace}~(ii), {we have}
			\[\dc \lambda_i=\max_{\vdc x\in \dc U} \ \|\vdc x\|^{-2}(\vdc x^*A(\Phi) \vdc x)\quad \text{and}\quad \dc \mu_i=\min_{\vdc x\in \dc V} \ \|\vdc x\|^{-2}(\vdc x^*A(\Phi[S]) \vdc x).\]
			
			Furthermore,  the following system
			\[\begin{bmatrix}
				U_s& O & -W_s & O \\
				U_s & U_d & -W_s & -W_d
			\end{bmatrix}
			\begin{bmatrix}
				{\vv}_s \\  {\vv}_d \\ {\vz}_s \\ {\vz}_d
			\end{bmatrix} =
			\begin{bmatrix}
				O\\O
			\end{bmatrix}
			\]
			always has solutions since  the size of its coefficient matrix  is {$2n\times (2n+2)$.}  {Let $\vv = \vv_s + \vv_d\epsilon$ and $\vz = \vz_s + \vz_d\epsilon$. Then,}
			there exists $\vdc w = \dc U{\vv} = \dc W {\vz} \in \dc U \cap \dc W$.
			Suppose that $\vdc w=\begin{pmatrix} {\vy}\\ 0 \end{pmatrix}$.
			Then we derive that
			\[\dc \lambda_i \ge \|\vdc w\|^{-2}(\vdc w^*A(\Phi)  \vdc w) = \|{\vy}\|^{-2}\left({\vy}^*A(\Phi[S])  {\vy}\right)\ge \mu_i.\]
			This proves the first part in \eqref{equ:interlacing}.
			
			The second inequality in \eqref{equ:interlacing} can be derived by choosing
			\[\dc U=\text{span}(\vdc x_1,\dots,\vdc x_{n+i-k}), \ \dc V=\text{span}(\vdc y_i,\dots,\vdc y_k),\  \dc W=\left\{\left(\begin{array}{c}
				\dc v\\ 0
			\end{array}\right)\in\dgset^n, \dc v\in \dc V\right\},\]
			and we do not repeat the details here.}
	\end{proof}	
	

	\begin{Cor}
		Let $\Phi=(G,\varphi)$ be a $\udset$-gain graph with $n$ vertices. For any vertex $v\in V$, the {eigenvalues} of $A(\Phi)$ and of $A(\Phi-v)$ are labeled in the decreasing order interlace as follows.
		\[\lambda_1(\Phi) \ge \lambda_1(\Phi-v) \ge \lambda_2(\Phi) \ge \lambda_2(\Phi-v)\ge\cdots \ge \lambda_{n-1}(\Phi-v) \ge \lambda_n(\Phi).\]
	\end{Cor}
	
	The following proposition presents some results for the eigenvalues and eigenvectors of the adjacency matrix of a  dual unit  gain graph.
	
	\begin{Prop}\label{lem:eig_adjac}
		Let $\Phi=(G,\varphi)$ be a $\udset$-gain graph with $n$ vertices and $\dc A=A(\Phi)$ be the adjacency matrix.
		Suppose $\dc \lambda=\lambda_s+\lambda_d \epsilon$ is an eigenvalue of $\dc A$ and $\vdc x={\vx}_s+{\vx}_d \epsilon= (\dc x_{1},\dots,\dc x_{n})^\top$ is its corresponding unit eigenvector. Then the following results hold.
		\begin{itemize}
			\item[(i)] The eigenvalue satisfies
			$$\dc \lambda=\vdc x^*\dcm A \vdc x=\sum_{e_{ij}\in E(\Phi)} 2Re({\dc x_i^*} \varphi(e_{ij})\dc x_j). $$
			\item[(ii)]  $\lambda_s$ is an  eigenvalue of  the matrix  $A_s$ with an  eigenvector ${\vx}_s$.
			\item[(iii)] The dual part satisfies $\lambda_d={\vx}_s^*A_d {\vx}_s$. Furthermore, if $A_s$ has $n$ simple {eigenvalues} $\lambda_{s,i}$ with associated unit {eigenvectors} $x_{s,i}$'s. Then $\dc A$ has exactly $n$ eigenvalues $\dc \lambda_i=\lambda_{s,i} +\lambda_{d,i}\epsilon$ with associated eigenvectors ${\vx}_i={\vx}_{s,i} +{\vx}_{d,i}\epsilon$, $i=1,\dots,n$, where
			\[\lambda_d={\vx}_s^*A_d {\vx}_s \ \text{and} \ {\vx}_{d,i} = \sum_{j\neq i} \frac{{\vx}_{s,j} {\vx}_{s,j}^* (A_d-\lambda_{d,i} I_n){\vx}_{s,i}}{\lambda_{s,i}-\lambda_{s,j}}. \]
			
			\item[(iv)] Two eigenvectors of $\dcm A$, associated with 	two eigenvalues with distinct standard parts, are orthogonal to each other.
		\end{itemize}
	\end{Prop}
	These results follow directly from \cite{QL23} and we omit the details of {their proofs} here.

	In general, the spectral radius of a  dual  matrix  is not well defined since a  dual matrix  may have no eigenvalue at all or have infinitely many eigenvalues. When the  dual  matrix  is Hermitian, all {of its} eigenvalues are dual numbers and we are able to define {their} order by \cite{QLY22}. Here, let  {$\rho_A(G)$} be the adjacency spectral radius of the underlying graph $G$ and  {$\rho_A(\Phi)$} be the adjacency spectral radius of the unit gain graph $\Phi$, respectively.
	The following theorem generalizes that of signed graphs \cite{CDD21, St19} and {complex} unit gain graphs {\cite{MKS22}}.
	
	\begin{Thm}\label{thm:spectrum_adj}
		Let $\Phi=(G,\varphi)$ be a $\udset$-gain graph. Then
		\begin{equation}\label{equ:rho_Phi<G}
			\rho_A(\Phi) \le \rho_A(G).
		\end{equation}
		Furthermore, if $\Phi$ is connected, then ${\lambda_1}(\Phi)=\rho_A(G)$ {(resp. $-\lambda_n(\Phi)=\rho_A(G)$)} if and only if $\Phi$ is balanced {(resp. antibalanced)}.
		{Here, {$\lambda_1(\Phi)$} {and {$\lambda_n(\Phi)$}} are the largest {and the smallest} eigenvalues of $A(\Phi)$, respectively.}
	\end{Thm}
	\begin{proof} Denote {$A = A(G)$, and} $B = (B_{ij}) = A(\Phi) = B_s + B_d\epsilon$ with $B_{ij} = {B_{ijs} + B_{ijd}\epsilon =} \varphi(e_{ij})$.
		Suppose $\vdc x\in {\dgset}^n$ is a unit eigenvector {of $B$} corresponding to an eigenvalue   {$\lambda$ of  $B$}. Let $ {\bf y}=|\vdc x|$. Then we have
		\begin{eqnarray}
			\nonumber	|\lambda| &=&  |\vdc x^*{B} \vdc x| \\
			\nonumber		&=& 2{\left|\sum_{{e_{ij}}\in E(\Phi)}  Re({\dc x_i^*} \varphi(e_{ij})\dc x_j)\right|} \\
			\label{equ:rho_Phi_G}	&\le &   2\sum_{{e_{ij}}\in E(\Phi)}  |{\dc x_i^*} \varphi(e_{ij})\dc x_j| \\
			\nonumber		& =  &  2 \sum_{{e_{ij}}\in E(\Phi)}  |{\dc x_i^*}|| \varphi(e_{ij})||\dc x_j| \\
			\nonumber	&=&   2 \sum_{{e_{ij}}\in E(\Phi)}  |{\dc x_i^*}| |\dc x_j| \\
			\nonumber	&=&  {\bf y}^* A(G)   {\bf y},
		\end{eqnarray}
		where the first and the second equalities follow from  {Proposition}~\ref{lem:eig_adjac}\,(i), the first inequality and the third equality follow  from Lemma~\ref{prelim_dc},  the fourth equality follows from $ \varphi(e_{ij})\in\udcset$ for any ${e_{ij}}\in {E(\Phi)}$.
		Therefore,
		\[St(|\dc  {\lambda}|) \le St({\bf y}^* A(G)   {\bf y})\le  \rho_A( G).\]
		If $St(|{\lambda}|) < \rho_A( G)$ for  {any eigenvalue $\lambda$ of $B$}, then $\rho_A(\Phi) < \rho_A(G)$.
		Otherwise, if there exists an eigenvalue $\lambda$ of   {$B$ satisfying} $St(|{\lambda|}) = St({\bf y}^* A(G)   {\bf y})=\rho_A( G)$, we claim that ${\lambda_d}=0$.

		By Proposition~\ref{lem:eig_adjac},  {we have}
		\[{\lambda_d} = \vx_s^*B_d\vx_s= \sum_{{e_{ij}}\in E(\Phi)} \dc x_{is}^* B_{ijd}\dc x_{js} = \sum_{{e_{ij}}\in E(\Phi)\cap \mathcal A_x} \dc x_{is}^* B_{ijd}\dc x_{js},\]
		where $\mathcal A_x = \{{e_{ij}}\,| \, x_{is} \neq 0, \ x_{js} \neq 0\}$.
		Since $St(|{\lambda}|) = St({\bf y}^* A(G)   {\bf y})$,  the equality of the standard part holds true in \eqref{equ:rho_Phi_G}. Therefore, for all ${e_{ij}}\in E(\Phi)\cap\mathcal A_x$, {we have}
		\[B_{ijs} = \frac{x_{is}}{|x_{is}|}\frac{x_{js}^*}{|x_{js}|}.\]
		Since each element in $B$ is a unit dual complex number, there is 	$B_{ijs}^*B_{ijd} + B_{ijd}^*B_{ijs}=0$ for all ${e_{ij}}\in E(\Phi)$.
		Therefore,
		\begin{eqnarray*}
			{\lambda_d} &=& \sum_{{e_{ij}}\in E(\Phi)\cap \mathcal A_x} \dc x_{is}^* B_{ijd}\dc x_{js} \\
			&=& \sum_{{e_{ij}}\in E(\Phi)\cap \mathcal A_x}   B_{ijd} B_{ijs}^* |x_{is}| |x_{js}| \\
			&=& \frac12 \sum_{{e_{ij}}\in E(\Phi)\cap \mathcal A_x}   (B_{ijd} B_{ijs}^* +B_{jid} B_{jis}^*)|x_{is}| |x_{js}| \\
			&=& \frac12 \sum_{{e_{ij}}\in E(\Phi)\cap \mathcal A_x}   (B_{ijd} B_{ijs}^* +B_{ijd}^* B_{ijs})|x_{is}| |x_{js}| \\
			&=& 0.
		\end{eqnarray*} 	
		Here, the {third} equality repeats all edges twice, the fourth equality follows from {the fact that} $B$ is Hermitian, and {the} last equality {holds} {because} each element in $B$ is a unit dual complex number.
		This proves {$|\lambda| \le \rho_A(G)$ for any eigenvalue $\lambda$ of $A(\Phi)$. Thus,} $\rho_A(\Phi) \le \rho_A(G)$.

		We now prove the second conclusion of this theorem.  Since $\Phi$ is connected, $A(G)$ is a {real} irreducible nonnegative matrix. It follows from the Perron-Frobenius Theorem that $\rho_A(G)$ is an eigenvalue of $A(G)$, i.e.,    the largest eigenvalue $\lambda_1(G) = \rho_A(G)$.
		Furthermore, the eigenvector corresponding to {$\lambda_1(G)$} is {real and} positive.

		On one hand, suppose that $\lambda_1(\Phi) = \rho_A(G)$, then   {the    inequality \eqref{equ:rho_Phi_G}} holds with equality {and ${\bf y}^* A(G)   {\bf y}= \rho_A(G)$}.  {This} implies that
		$$c_{ij}:={\dc x_i^*}\varphi(e_{ij})\dc x_j = |{\dc x_i^*}| |\dc x_j|$$
		is a positive real number for all ${e_{ij}}\in E$ {and ${\bf y}_s$} is the Perron-Frobenius eigenvector of $A(G)$.
		Since {${\vy_s}$} is a positive vector,   {$|x_i|>0$} and $\dc x_i$ is invertible for all $i=1,\dots,n$.
		Therefore,
		\[\varphi(e_{ij}) = ({\dc x_i^*})^{-1}c_{ij}\dc x_j^{-1} = c_{ij} \frac{{\dc x_i}}{|\dc x_i|^2}\frac{{\dc x_j^*}}{|\dc x_j|^2}= \frac{{\dc x_i}}{|\dc x_i|}\frac{{\dc x_j^*}}{|\dc x_j|},\]
		where the last equality follows from $|\varphi(e_{ij})|=1$.
		{In {the} other words},
		\[\theta(v_i) = \frac{{\dc x_i^*}}{|\dc x_i|}\]
		is a potential function of $\Phi$. This verifies that $\Phi$ is balanced.
		
		{On the other hand, s}uppose $\Phi$ is balanced. It follows from  {Proposition}~\ref{lem:similar} that $A(\Phi)$ and $A(G)$ share the same set of eigenvalues. Therefore,  $\lambda_1(G)=\lambda_1(\Phi)$ and
		$$\rho_A(\Phi) \le \rho_A(G)  =\lambda_1(G)=\lambda_1(\Phi)\le \rho_A(\Phi),$$
		where the first inequality follows from  \eqref{equ:rho_Phi<G} and the second equality  follows from the  Perron-Frobenius Theorem.
		Hence, $\lambda_1(\Phi)=\rho_A(G)$.

		{Finally, the last conclusion follows from $-\lambda_n(\Phi) = \lambda_1(-\Phi)$  and $\Phi$ is antibalanced if and only if $-\Phi$ is balanced.}
		
		This completes the proof.
	\end{proof}
	
	
	The following result follows partially from the Gershgorian type theorem for dual quaternion Hermitian matrices in \cite{QWL23}.
	\begin{Prop}
		Let $\Phi=(G,\varphi)$ be a $\udset$-gain graph. Then
		\[{\rho_A(\Phi)} \le \Delta.\]
		{The equality holds if and only if $G$ is $\Delta$-regular and either $\Phi$ or $-\Phi$ is balanced.}
		{Here, $\Delta$ is the maximum vertex degree of $G$.}
	\end{Prop}
	\begin{proof}
		Suppose $\dc \lambda \in\dgset$ is an eigenvalue of $A(\Phi)$ and $\vdc x$ is its corresponding eigenvector; namely, $A(\Phi)\vdc x = \lambda\vdc x$.
		Let $i\in\arg\max_l |\dc x_l|$.  Then we have
		\[ |\lambda||\dc x_i| = |\lambda\dc x_i| = {\left|\sum_{j\neq i} \varphi(e_{ij})\dc x_j\right|} \le \sum_{j\neq i} |\varphi(e_{ij})||\dc x_j|.\]
		Therefore,
		\[|\dc \lambda| \le \max_i \sum_{j\neq i} |\varphi(e_{ij})| =\max_i \ d_i = \Delta.\]
		{Furthermore, it is well-known that, if $G$ is a simple connected graph with maximum vertex degree $\Delta$, then $\rho_A(G) \le \Delta$, and equality holds if and only if $G$ is regular.   By Theorem \ref{thm:spectrum_adj}, the equality in this proposition holds if and only if $G$ is $\Delta$-regular and either $\Phi$ or $-\Phi$ is balanced.}
		This completes the proof.	
	\end{proof}

	\section{Coefficient Theorems of Characteristic Polynomials}

		Let $A=(a_{ij})$ be a Hermitian matrix in $\gset^{n\times n}$ and $\sigma$ be a permutation of $S_n=\{1,\dots,n\}$.
		Write $\sigma$ as a product of disjoint cycles, i.e.,
		\[\sigma=(n_{11}\cdots n_{1l_1})(n_{21}\cdots n_{2l_2})\cdots(n_{r1}\cdots n_{rl_r}),\]
		where $n_{i1}<n_{ij}$ for all $j>1$ and $n_{11}>n_{21}>\cdots>n_{r1}$  for each $i=1,\dots,r$.
		Then 	the {\bf Moore determinant} for quaternion matrices  defined by Eliakim Hastings Moore \cite{Moore22} is written as follows,
		\begin{equation}\label{equ:Mdet}
			\mathrm{Mdet}(A) = \sum_{\sigma\in S_n} |\sigma|a_{\sigma},
		\end{equation}
		where $|\sigma|$ denotes the parity of $\sigma$ and
		\[a_{\sigma} = (a_{n_{11},n_{12}}a_{n_{12},n_{13}}\cdots a_{n_{1l_1},n_{11}})(a_{n_{21},n_{22}}\cdots a_{n_{2l_2},n_{21}})\cdots(\cdots a_{n_{rl_r},n_{r1}}).\]
		For more results on the Moore determinant, please refer to \cite{Asl96, Kyr12}.
		
		The adjacency and Laplacian matrices of a unit gain graph are Hermitian matrices.  To study their characteristic polynomials, we need to consider the determinants of these matrices. When the Hermitian matrix is complex or real, we may use the ordinary determinant.  However, if the matrix is quaternion, as the multiplication is not commutative, the ordinary determinant is not well-defined.  Then,  the Moore determinant may be used as in \cite{BBCRS22}.
		

		Recently, \cite{CQ24}   studied the Moore determinant of dual Hermitian matrices. Given a dual Hermitian matrix $A\in\dgset^{n\times n}$,  $\mathrm{Mdet}(A) =0$ if and only if $A_s$ is singular and there exist at least one zero or two infinitesimal eigenvalues.
		The value  $\mathrm{Mdet}(A)$ equals to the product of the   eigenvalues of $A$.   These pave  a way for studying characteristic polynomials of such dual Hermitian matrices.

	\bigskip
	
	We now    {study the Coefficient Theorem for {dual complex and  quaternion unit gain graphs}.
		This extends the  Coefficient Theorem of simple graphs \cite{CDS95},  {signed} graphs \cite{BCKW18}, and  {$\ucset$}-gain graphs \cite{MKS22}.
		We begin with the following lemma.
		
		\begin{Lem}\label{lem:cyc2realgain}
			Let $\Phi=(G,\varphi)$ be a $\udset$-gain graph and let $\mathcal C(G)$ be the set of cycles of $G$. Then the following function
			\begin{equation}\label{equ:realgainCycles}
				\mathcal R: C\in\mathcal C(G) \rightarrow Re(\varphi(\vec{C}_v))\in\drset
			\end{equation}
			is well defined and independent of the choice of  $v\in C$ and the direction of the directed cycle $\vec{C}_v$.
		\end{Lem}
		\begin{proof}
			Let $v_1=v$, $\vec{C}_v=v_1 e_{12} v_2\cdots v_l e_{l1} v_1$ and $\vec{C}_v^{-1}=v_1 e_{1 l} v_l\cdots v_2 e_{21} v_1$.
			Then $\varphi(\vec{C}_v), \varphi(\vec{C}_v^{-1}) \in\udset$ and $\varphi(\vec{C}_v)\varphi(\vec{C}_v^{-1})=1$.
			Therefore, the two dual unit elements $\varphi(\vec{C}_v)$ and $\varphi(\vec{C}_v^{-1})$ are conjugate and share the same real parts.
			In other words, $\mathcal R(\vec C_v)$ is independent of the direction of the   cycle.
			
			Furthermore, let $w=v_t$ with $t\in \{2,\dots,l\}$ and  $\vec{C}_w=v_t e_{t,t+1} v_{t+1}\cdots v_{t-1} e_{t-1,t} v_t$.
			Denote $q_1=\varphi(e_{12})\cdots \varphi(e_{t-1,t})$ and  $q_2=\varphi(e_{t,t+1})\cdots \varphi(e_{l,1})$. Then there is $\varphi(\vec{C}_v)=q_1q_2$,  $\varphi(\vec{C}_w)=q_2q_1$, and
			\[Re(\varphi(\vec{C}_v))=Re(q_1q_2)=Re(q_2q_1)=Re(\varphi(\vec{C}_w)).\]
			The proof is completed.
		\end{proof}
		
		Let $C_r$ denote the cycle of order $r$ and let $K_r$ be the complete graph with $r$ vertices.
		An elementary  graph  is any graph in the set $\{K_2,C_r (r\ge3)\}$ and a basic  graph  is the disjoint union of elementary  graphs.
		Suppose $B$  is a basic  graph. Denote  $\mathcal C(B)$ as the class of cycles in $B$,  $p(B)$ as the number of components of $B$.
		Let $\mathcal B_n(G)$ be the set of subgraphs of $G$ that are basic graphs with $n$ vertices and
		\[c(B) = |\mathcal C(B)|,\quad \mathcal R(B)=\Pi_{C\in\mathcal C(B)} \mathcal R(C).\]
		Here, $\mathcal R(\cdot)$ is defined by \eqref{equ:realgainCycles}.
		
		\begin{Thm}[Coefficient Theorem]
			Let $\Phi=(G,\varphi)$ be a $\udset$-gain graph with $n$ vertices. Then
			\[\mathrm{Mdet} (A(\Phi)) = \sum_{B\in\mathcal B_n(G)} (-1)^{n+p(B)} 2^{c(B)}\mathcal R(B). \]
			Furthermore, let 	 			 $p_{A(\Phi)}(x)=x^n+\sum_{i=1}^nc_i x^{n-i}$
			be the characteristic polynomial of $A(\Phi)$. Then
			\[c_i=\sum_{B\in\mathcal B_i(G)}(-1)^{p(B)}2^{c(B)}\mathcal R(B).\]
		\end{Thm}
		\begin{proof}
			By Lemma~\ref{lem:cyc2realgain}, the results can be derived following the same proof with that of \cite{BBCRS22} and we omit the details  here.
			%
			%
		\end{proof}
		
			%
	}
	
	As the eigenvalues of a dual Hermitian matrix are dual numbers, the coefficients of the characteristic polynomial of a dual Hermitian matrix are also dual numbers.   It is difficult to handle such a polynomial.  Furthermore, the roots of such a characteristic polynomial are not necessary to be eigenvalues of the dual Hermitian matrix.
	{For instance, the roots of $(\lambda-1)^2=0$ are $1+b\epsilon$ for all $b\in \rset$.} See \cite{QC24} {for more details}.
	
	We may take another approach to solve this problem.   Suppose that $A(\Phi) = A_s(\Phi) + A_d(\Phi)\epsilon$. Then we may calculate the coefficients of the characteristic polynomial of $A_s(\Phi)$ by the coefficient theorem for complex or quaternion unit gain graphs.   Note that this polynomial is of real coefficients as $A_s(\Phi)$ is a complex or quaternion Hermitian matrix and if $\lambda_s$ is a root of this characteristic polynomial, then it is the standard part of an  eigenvalue $\lambda$ of $A(\Phi)$. If $\lambda_s$ is a single root of this characteristic polynomial, then by \cite{QC24}, we have $\lambda = \lambda_s + \lambda_d\epsilon$, where $\lambda_d = \vx_s^*A_d(\Phi)\vx_s$ and $\vx_s$ is a unit eigenvector of $A_s(\Phi)$, associated with $\lambda_s$.   If $\lambda_s$ is a $k$-multiple root of this characteristic polynomial, then by \cite{QC24}, there are $k$ eigenvalues $\lambda_i = \lambda_s + \lambda_{di}\epsilon$ of $A(\Phi)$, for $i = 1,\dots, k$, where $\lambda_{d1},\dots, \lambda_{dk}$ are eigenvalues of the {\bf supplement matrix} $W^*A_d(\Phi)W$, $W = (\vv_1, \dots, \vv_k)$ and  $\vv_1,\dots, \vv_k$ are $k$ orthonormal eigenvectors of $A_s$, associated with $\lambda_s$.
Note that the supplement matrix $W^*A_d(\Phi)W$ is a complex or quaternion Hermitian matrix.
		Its elements are not unit elements.  Furthermore, $W^*A_d(\Phi)W$ may not keep the special structure of $A_d$. In this way, may we still use the coefficient theorem to calculate their coefficients?

	\section{Eigenvalues of Laplacian Matrices}
	
	Similar {to} Proposition~\ref{lem:similar}, the switching class has also a unique Laplacian spectrum.

	\begin{Prop}\label{lem:similar_Laplace}
		Let $\Phi_1=(G,\varphi_1)$ and $\Phi_2=(G,\varphi_2)$ {be both} $\udset$-gain graphs. If $\Phi_1\sim \Phi_2$, then {$\Phi_1$ and $\Phi_2$} have the same {Laplacian} spectrum. That is, $\sigma_L(\Phi_1) = \sigma_L(\Phi_2)$.
	\end{Prop}
	
	The interlacing theorem of the adjacency matrix in Theorem~{\ref{Thm:interlacing_adj}}  also holds  for the Laplacian matrix.
	\begin{Thm}(Interlacing Theorem of the Laplacian Matrix) \label{Thm:interlacing_Lap} Let $\Phi=(G,\varphi)$ be a $\udset$-gain graph with $n$ vertices and $S$ be a subset of $V$ with $k$ vertices.
		Denote the eigenvalues of $L(\Phi)$ and  $L(\Phi[S])$ by
		\[{\dc \lambda_1\ge \dc  \lambda_2\cdots \ge \dc \lambda_n  \quad \text{and} \quad \dc  \mu_1\ge \dc \mu_2\ge \cdots \ge \dc  \mu_k,}\]
		respectively.
		Then the following inequalities hold:
		\begin{equation*}
			\dc \lambda_i\ge \dc  \mu_i\ge \dc \lambda_{n+i-k},\quad \forall\, 1\le i\le k.
		\end{equation*}
	\end{Thm}
	The proof is similar to the proof of Theorem~{\ref{Thm:interlacing_adj}}  and we omit the  details of the proof here.
	
	{Let the signless Laplacian matrix be $Q(G)=D(G)+A(G)$ and $\Phi^{1}= (G,-1)$ be the $\udcset$-gain graph with all gains $-1$.}
	{Again, since the Laplacian matrix of a dual complex unit gain graph $\Phi$ is Hermitian, we may denote the spectral {radius} of that  Laplacian matrix {(resp. signless Laplacian matrix) of $\Phi$} as $\rho_L(\Phi)$ {(resp. $\rho_Q(\Phi)$)}, the {spectral} radius of the Laplacian matrix {(resp. signless Laplacian matrix)} of its underlying graph $G$ as $\rho_L(G)$ {(resp. $\rho_Q(G)$)}, {respectively}, and discuss their properties.}
	{Then we have the following theorem.}
	
	\begin{Thm}\label{thm:spectrum_Lap}
		Let $\Phi=(G,\varphi)$ be a $\udset$-gain graph and $\Phi^1=(G,-1)$. Then
		\begin{equation*}
			\rho_L(\Phi) \le {\rho_L(\Phi^1)=\rho_Q(G)}.
		\end{equation*}
		Furthermore, if $\Phi$ is connected, {then $\rho_L(\Phi)=\rho_Q(G)$} if and only if
		{$\Phi\sim \Phi^1$.}
	\end{Thm}
	
	\begin{Prop}
		Let $\Phi=(G,\varphi)$ be a $\udset$-gain graph. Then
		\[{\rho_L(\Phi)} \le 2\Delta.\]
		{Here, $\Delta$ is the maximum vertex degree of $G$.}
		{Moreover, the equality holds if and only if $G$ is $\Delta$-regular and $\Phi\sim \Phi^1$.}
	\end{Prop}
	\begin{proof}
		{Suppose $\lambda \in\dgset$ is an eigenvalue of $L(\Phi)$ and $\vdc x$ is its corresponding eigenvector; namely, $L(\Phi)\vdc x = \lambda\vdc x$.
			Let $i\in\arg\max_l |\dc x_l|$.  Then we have
			\[ |\lambda||\dc x_i| = |\lambda\dc x_i| = {\left|d_i\dc x_i+\sum_{j\neq i} \varphi(e_{ij})\dc x_j\right|} \le d_i|\dc x_i|+\sum_{j\neq i} |\varphi(e_{ij})||\dc x_j|.\]
			Therefore,
			\[|\dc \lambda| \le \max_i d_i+ \sum_{j\neq i} |\varphi(e_{ij})| =\max_i \ 2d_i = 2\Delta.\]
			
			Furthermore, it is well-known that, if $G$ is a simple connected graph with maximum vertex degree $\Delta$, then $\rho_Q(G) \le 2\Delta$, and equality holds if and only if $G$ is regular.   By Theorem \ref{thm:spectrum_Lap}, the equality in this proposition holds if and only if   $G$ is $\Delta$-regular and $\Phi\sim \Phi^1$.
			
			This completes the proof.}
	\end{proof}
	
	%
	%
	%
	%
	%
	
	\bigskip

	{\bf Acknowledgment}
	We are thankful to Professors Zhuoheng He  and Guangjing Song for their help.

	\bigskip
	
	{\section*{Declarations}
		
		{\bf Funding} This research is supported by the R\&D project of Pazhou Lab (Huangpu) (No. 2023K0603), the National Natural Science Foundation of China   (No. 12371348), and the Fundamental Research Funds for the Central Universities (Grant No. YWF-22-T-204).
		
		\noindent
		{\bf Conflicts of Interest} The author declares no conflict of interest.
		
		\noindent
		{\bf Data Availability} Data will be made available on reasonable request.
	}


\end{document}